\documentclass[12pt]{amsart}
 
\usepackage{amssymb, amscd, txfonts}
\usepackage{graphicx}
\usepackage[all]{xy} 
\usepackage{mathrsfs}
 
\numberwithin{equation}{section}

\newtheorem{theorem}{Theorem}[section]
\newtheorem{proposition}[theorem]{Proposition}
\newtheorem{lemma}[theorem]{Lemma}
\newtheorem{corollary}[theorem]{Corollary}
\newtheorem{step}{Step}

\theoremstyle{definition}
\newtheorem{definition}[theorem]{Definition}
\newtheorem{example}[theorem]{Example}

\theoremstyle{remark}
\newtheorem{remark}[theorem]{Remark}

\renewcommand{\ker}{\operatorname{Ker}}

\newcommand{\Z}{\mathbb{Z}}
\newcommand{\Q}{\mathbb{Q}}
\newcommand{\R}{\mathbb{R}}
\newcommand{\C}{\mathbb{C}}
\newcommand{\proj}{{\mathbb P}}

\newcommand{\HH}{\mathbb{H}}

\newcommand{\G}{\Gamma}
\newcommand{\LQ}{\Lambda_{\mathbb{Q}}}
\newcommand{\D}{\mathcal{D}}
\newcommand{\DL}{\mathcal{D}_{\Lambda}}
\newcommand{\Sp}{{\rm Sp}(\Lambda)}
\newcommand{\SL}{{\rm SL}_{2}}
\newcommand{\OL}{{\rm O}^{+}(\Lambda)}
\newcommand{\CHtor}{{\rm CH}_{0}(X^{tor})}
\newcommand{\CHtorQ}{{\rm CH}_{0}(X^{tor})_{\mathbb{Q}}}

\newcommand{\GIZ}{\Gamma(I)_{\mathbb{Z}}}
\newcommand{\UIQ}{U(I)_{\mathbb{Q}}}
\newcommand{\UIZ}{U(I)_{\mathbb{Z}}}
\newcommand{\GIZbar}{\overline{\Gamma(I)}_{\mathbb{Z}}}

\begin{document}

%%%%%%% Title %%%%%%%%%%%%%%%%%%%%%%%%
\title[]{Zero-cycles over zero-dimensional cusps}
\author[]{Shouhei Ma}
\thanks{Supported by KAKENHI 21H00971 and 20H00112.} 
\address{Department~of~Mathematics, Tokyo~Institute~of~Technology, Tokyo 152-8551, Japan}
\email{ma@math.titech.ac.jp}
%\dedicatory{}

\begin{abstract}
We prove that all points of a toroidal compactification lying over $0$-dimensional cusps 
are rationally equivalent in the integral Chow group 
for most classical modular varieties 
(Siegel, Hilbert, orthogonal, Hermitian, quaternionic). 
This gives a generalization, and even strengthening, of the Manin-Drinfeld theorem in higher dimension 
from the viewpoint of algebraic cycles. 
The same result no longer holds for Picard modular varieties, 
but for them we prove that the difference of any two "special" boundary points, 
which are dense in the boundary, is torsion. 
\end{abstract} 

\maketitle

%%%%Introduction
\section{Introduction}\label{sec: intro}

In \cite{Manin}, \cite{Dr}, Manin and Drinfeld discovered that 
the difference of two cusps of a congruence modular curve is torsion in its Picard group. 
We wish to study higher dimensional analogue of the Manin-Drinfeld theorem 
from the viewpoint of algebraic cycles. 
In \cite{Ma} the Manin-Drinfeld theorem was generalized in the form of 
rational equivalence of two cusps in the rational Chow group of 
the Baily-Borel compactification of some modular varieties. 
The Baily-Borel compactification is however highly singular at the boundary in general, 
and it has been recognized that the study of Chow groups of singular varieties 
is more difficult than for nonsingular ones.  

Toroidal compactification \cite{AMRT} provides an explicit desingularization 
of the Baily-Borel compactification when the arithmetic group is neat. 
Even when not smooth, its singularities are relatively mild, and its boundary is of codimension $1$. 
%A natural question is then, what kind of boundary cycles we should look at. 
It should be noticed, however, that 
a naive analogue of the Manin-Drinfeld theorem no longer holds for the boundary \textit{divisors} 
of toroidal compactification. 
From the viewpoint of algebraic cycles, one of the next cycles to look at will be $0$-cycles. 
In this paper we study some boundary $0$-cycles of toroidal compactification. 
For classical modular varieties with tube domain model, 
we prove that they are rationally equivalent in the \textit{integral} Chow group, 
hence define a distinguished integral Chow class.   
The viewpoint of $0$-cycles thus provides one direction in which 
the Manin-Drinfeld theorem can be generalized and even strengthened. 

%They thus give rise to a distinguished element in the integral Chow group of the toroidal compactification.  

To be more specific, we consider the following series of modular varieties. 
\begin{itemize}
\item Siegel modular variety $\mathcal{D}/{\G}$ 
attached to a finite-index subgroup ${\G}$ of the symplectic group ${\Sp}$ 
of a symplectic lattice $\Lambda$ of rank $>2$. 
\item Hilbert modular variety ${\HH}^{n}/{\G}$ 
attached to a finite-index subgroup ${\G}$ of ${\SL}(\mathcal{O}_{F})$  
for a totally real number field $F$ of degree $n>1$. 
\item Orthogonal modular variety $\mathcal{D}/{\G}$ 
attached to a finite-index subgroup ${\G}$ of the orthogonal group ${\OL}$ of 
an integral quadratic form $\Lambda$ of signature $(2, n)$ with $n>2$. 
%We assume that either $n\geq 3$ or $\Lambda$ has Witt index $1$. 
\item We also cover Hermitian modular varieties attached to the unitary groups ${\rm U}(n, n)$ and 
quaternionic modular varieties attached to the groups ${\rm SO}^{\ast}(4n)$, both with $n>1$. 
\end{itemize}

Except for the Hilbert case, ${\D}$ is the classical irreducible Hermitian symmetric domain of type 
$III, IV, I_{n,n}, II_{2n}$ respectively. 
According to \cite{PS}, they are characterized by the existence of tube domain model. 
In the first two cases ${\G}$ is automatically a congruence subgroup (\cite{Me}, \cite{BLS}, \cite{Se}). 

Let $X^{bb}$ be the Baily-Borel compactification of $Y=\mathcal{D}/{\G}$, 
and $X^{tor}$ be the toroidal compactification of $Y$ defined by a ${\G}$-admissible collection of fans $\Sigma$. 
(We suppress the dependence on $\Sigma$ in the notation.) 
We may assume that $\Sigma$ is chosen so that $X^{tor}$ is projective (\cite{AMRT}). 
We have a natural morphism $\pi \colon X^{tor}\to X^{bb}$. 
In the Hilbert case, $X^{bb}$ has only $0$-dimensional cusps. 

Our main result is the following. 

\begin{theorem}\label{main} 
Let $Y$ be one of the modular varieties as above and 
$X^{tor}$ be a projective toroidal compactification of $Y$. 
Let $P, Q\in X^{tor}$ be two points such that $\pi(P), \pi(Q)$ are $0$-dimensional cusps (which may be distinct). 
Then $[P]=[Q]$ in the integral Chow group ${\CHtor}$ of $X^{tor}$. 
\end{theorem}

%\begin{theorem}[Siegel case]\label{thm: Siegel}
%Let $X^{tor}$ be a toroidal compactification of a Siegel modular variety. 
%Let $P, Q\in X^{tor}$ be two points such that $\pi(P), \pi(Q)$ are $0$-dimensional cusps (which may be distinct). 
%Then $[P]=[Q]$ in the integral Chow group ${\CHtor}$ of $X^{tor}$. 
%\end{theorem}

%\begin{theorem}[Hilbert case]\label{thm: Hilbert}
%Let $X^{tor}$ be a toroidal compactification of a Hilbert modular variety $Y$. 
%Let $P, Q$ be two points of $X^{tor}-Y$. 
%Then $[P]=[Q]$ in ${\CHtor}$. 
%\end{theorem}

%\begin{theorem}[Orthogonal case]\label{thm: orthogonal}
%Let $X^{tor}$ be a toroidal compactification of an orthogonal modular variety. 
%Suppose that either $n\geq 3$ or $\Lambda$ has Witt index $1$. 
%Let $P, Q\in X^{tor}$ be two points such that $\pi(P), \pi(Q)$ are $0$-dimensional cusps. 
%Then $[P]=[Q]$ in ${\CHtor}$. 
%\end{theorem}

Theorem \ref{main} is in contrast to the case of modular curves, 
namely the original Manin-Drinfeld theorem. 
In that case, the difference of two cusps is nontrivial unless the modular curve has genus $0$, 
and the calculation of its order in the Picard group has been the subject of many works on 
modular units and cuspidal divisor class groups. 
Thus, in this aspect, the situation gets simpler in higher dimension. 

As a consequence of Theorem \ref{main}, we obtain a distinguished element of ${\CHtor}$ 
that will play a fundamental role in the Chow ring of $X^{tor}$. 
In a few examples where $X^{tor}$ is a $K3$ Hilbert modular surface, 
this Chow class coincides with the Beauville-Voisin class \cite{BV} of the $K3$ surface (Example \ref{K3}). 
A similar coincidence also holds for a $K3$ Picard modular surface (Example \ref{K3 II}). 

The proof of Theorem \ref{main} is relatively simple. 
In the Siegel case, it proceeds as follows. 
We first show that $[P]-[Q]$ is torsion in ${\CHtor}$. 
When $\pi(P)=\pi(Q)$, this is a consequence of the structure of the boundary. 
When $\pi(P)\ne \pi(Q)$, we apply the Manin-Drinfeld theorem to a chain of certain modular curves in $Y$ 
joining $\pi(P)$ and $\pi(Q)$. 
Finally, we conclude that $[P]=[Q]$ 
as a consequence of the Roitman theorem \cite{Ro2} and the Margulis normal subgroup theorem \cite{Margu}, 
two deep theorems from different branches of Mathematics. 
We can say that it is this last step 
that makes the difference between the curve case and the higher dimensional case. 
At the same time, this step leaves it somewhat mysterious 
\textit{how} $P$ and $Q$ are equivalent in the integral Chow group. 

The proof in the Hilbert case is similar. 
It is also possible to give a similar proof in the orthogonal case, 
but instead of such a repetition, 
we \textit{reduce} the proof to the cases of Siegel $3$-folds and Hilbert modular surfaces already proved. 
This argument enables us to prove Theorem \ref{main} in the orthogonal case without assuming ${\G}$ congruence subgroup. 
The cases of ${\rm U}(n, n)$ and ${\rm SO}^{\ast}(4n)$ are similar. 
All classical domains for which our proof of Theorem \ref{main} may work are covered. 

The case excluded from the orthogonal case is $n=2$ with Witt index $2$. 
This is essentially the case of product of two modular curves. 
Theorem \ref{main} does not hold for them. 
Indeed, their Albanese map is the product of the Abel-Jacobi maps of each curve, 
and so injective if both curves have genus $>0$. 
Thus the assumption in the orthogonal case would be best possible.

Theorem \ref{main} does not hold for Picard modular varieties attached to the unitary groups ${\rm U}(1, n)$. 
Perhaps this case could be more interesting. 
For simplicity we assume ${\G}$ neat (and congruence). 
The toroidal compactification of a Picard modular variety is canonical and has 
disjoint union of abelian varieties as the boundary. 
In general, its Albanese map does not contract the boundary to one point. 
Nevertheless we can prove the following analogue of Theorem \ref{main}. 
We define \textit{special} boundary points using the structure of toroidal compactification (Definition \ref{def: special point}). 
In each component of the boundary, such points form the set of torsion points 
of the abelian variety if we take one of them as the origin. 

\begin{theorem}\label{thm: Picard}
Let $X^{tor}$ be the toroidal compactification of a Picard modular variety $Y$. 
Let $P, Q\in X^{tor}-Y$ be two special boundary points. 
Then $[P]-[Q]$ is torsion in ${\CHtor}$. 
\end{theorem}

By Theorem \ref{thm: Picard}, 
the Chow class of special boundary points is uniquely determined in the rational Chow group ${\CHtorQ}$.  
By our geometric characterization of special boundary points, 
any intersection $0$-cycle of a modular curve and a boundary divisor is contained in 
the span of this Chow class in ${\CHtorQ}$ (Corollary \ref{cor: Chow class special point}). 

Theorem \ref{thm: Picard} leads to a (surviving) Chow-theoretic analogue of 
cuspidal divisor class groups for Picard modular varieties, 
as the subgroup of ${\CHtor}$ generated by the differences of two special boundary points 
(\S \ref{ssec: cusp group}). 
By Theorem \ref{thm: Picard} and Roitman's theorem \cite{Ro2}, 
to understand this group is reduced to 
the study of the configuration of the Albanese image of the boundary abelian varieties.

This paper is organized as follows. 
In \S \ref{sec: Siegel} we prove Theorem \ref{main} for Siegel modular varieties. 
In \S \ref{sec: Hilbert} we prove Theorem \ref{main} for Hilbert modular varieties. 
In \S \ref{sec: orthogonal} we prove Theorem \ref{main} for orthogonal modular varieties, 
and give an outline for the remaining cases. 
In \S \ref{sec: Picard} we prove Theorem \ref{thm: Picard}.

%%%%Siegel
\section{Siegel modular varieties}\label{sec: Siegel}

In this section, after recollection in \S \ref{ssec: Siegel recall}, 
we prove Theorem \ref{main} for Siegel modular varieties in \S \ref{ssec: Siegel proof}. 

\subsection{Siegel modular varieties}\label{ssec: Siegel recall}

We first recall Siegel modular varieties and their Satake and toroidal compactifications. 
Let $\Lambda$ be a free abelian group of rank $2g>2$ 
equipped with a nondegenerate symplectic form 
$\Lambda \times \Lambda \to {\Z}$. 
The Hermitian symmetric domain ${\D}={\DL}$ attached to $\Lambda$ 
is defined as the open locus of the Lagrangian Grassmannian  
consisting of $g$-dimensional isotropic subspaces $W$ of $\Lambda_{{\C}}$ 
such that the associated Hermitian form on $W$ is positive-definite. 
Let ${\G}$ be a finite-index subgroup of the symplectic group ${\Sp}$. 
By \cite{Me}, \cite{BLS}, ${\G}$ is a congruence subgroup. 
The quotient 
$Y={\D}/{\G}$ 
is a quasi-projective variety, called a Siegel modular variety. 

The Satake compactification of $Y$, denoted by $X^{bb}$ for consistency with other sections, 
is defined as the quotient by ${\G}$ of the union of ${\D}$ 
and its rational boundary components, equipped with the Satake topology. 
It is normal and projective (\cite{BB}). 
Rational boundary components (cusps) of ${\D}$ correspond to 
isotropic subspaces $I$ of ${\LQ}$. 
In particular, when $I$ is maximal, i.e., $\dim I = g$, 
the corresponding cusp is $0$-dimensional, 
given by the point $[I_{{\C}}]$ of the Lagrangian Grassmannian. 
%We call (the closure of) the image of a rational boundary component in $X^{bb}$ a \textit{cusp}. 

A toroidal compactification $X^{tor}=X_{\Gamma}^{\Sigma}$ of $Y$ (\cite{AMRT}) 
is constructed by choosing a ${\G}$-admissible collection of fans 
$\Sigma=(\Sigma_{I})$, 
one for each isotropic subspace $I$ of ${\LQ}$. 
Each fan $\Sigma_{I}$ is a polyhedral cone decomposition of 
the cone of positive semidefinite forms with rational kernel in ${\rm Sym}^{2}I_{{\R}}$. 
This defines a partial compactification $({\D}/{\UIZ})^{\Sigma_{I}}$ of ${\D}/{\UIZ}$, 
where ${\UIZ}={\G}\cap {\UIQ}$ and ${\UIQ}\simeq {\rm Sym}^{2}I$ is 
the center of the unipotent part of the stabilizer $N(I)_{{\Q}}$ of $I$ in ${\rm Sp}(\Lambda_{{\Q}})$. 
We write  
${\GIZ}={\G}\cap N(I)_{{\Q}}$ 
and ${\GIZbar}={\GIZ}/{\UIZ}$. 
Then $X^{tor}$ is constructed as a gluing of $Y$ and 
the quotients $({\D}/{\UIZ})^{\Sigma_{I}}/{\GIZbar}$. 
$X^{tor}$ is a compact Moishezon space with a natural morphism 
$\pi : X^{tor}\to X^{bb}$. 
We may choose $\Sigma$ so that $X^{tor}$ is projective (\cite{AMRT}). 

We are interested in the case $\dim I = g$. 
Let $p\in X^{bb}$ be the $0$-dimensional cusp corresponding to $I$. 
In this case ${\D}/{\UIZ}$ is embedded in the algebraic torus 
$T_{I}={\rm Sym}^{2}I_{{\C}}/{\UIZ}$ 
%as the locus of symmetric forms with positive-definite imaginary part, 
by the tube domain realization of ${\D}$ with respect to $I$ 
(= isomorphism with the Siegel upper half space). 
The partial compactification 
${\D}/{\UIZ}\hookrightarrow ({\D}/{\UIZ})^{\Sigma_{I}}$ 
is an open set of the torus embedding 
$T_{I}\hookrightarrow T_{I}^{\Sigma_{I}}$ 
defined by the fan $\Sigma_{I}$. 
If $\sigma\in \Sigma_{I}$ is a cone 
whose relative interior consists of positive-definite forms, 
the corresponding torus orbit 
${\rm orb}(\sigma)\subset T_{I}^{\sigma} \subset T_{I}^{\Sigma_{I}}$ 
is contained in $({\D}/{\UIZ})^{\Sigma_{I}}$. 
Then the locus of boundary points of $({\D}/{\UIZ})^{\Sigma_{I}}$ which lie over the cusp $p$ 
is the union of these torus orbits (\cite{AMRT} p.164).  
This shows that 
\begin{equation}\label{eqn: fiber 0dim cusp}
\pi^{-1}(p) = (\sqcup_{\sigma}{\rm orb}(\sigma))/{\GIZbar}, 
\end{equation}
where $\sigma$ runs over cones with positive-definite relative interior, 
and $\sqcup_{\sigma}{\rm orb}(\sigma)$ is the union of infinitely many toric varieties 
whose configuration is determined by $\Sigma_{I}$.

\subsection{Proof of Theorem \ref{main}}\label{ssec: Siegel proof}

We now prove Theorem \ref{main} for Siegel modular varieties. 
Recall that the given toroidal compactification $X^{tor}=X_{{\G}}^{\Sigma}$ 
is assumed to be projective. 
We begin with the following reduction, which will be necessary in Step \ref{step4}. 

\begin{step}\label{step1} 
We may assume that ${\G}$ is neat and $X^{tor}$ is smooth and projective. 
\end{step}

\begin{proof}
We take a neat subgroup $\Gamma'$ of ${\G}$ of finite index. 
Then $\Sigma$ is also $\Gamma'$-admisible, 
so we have the toroidal compactification 
$X_{\Gamma'}^{\Sigma}$ of ${\D}/\Gamma'$ 
defined by $\Sigma$. 
Since we have a finite morphism 
$X_{\Gamma'}^{\Sigma} \to X_{\Gamma}^{\Sigma}$ 
and $X_{\Gamma}^{\Sigma}$ is projective, 
$X_{\Gamma'}^{\Sigma}$ is projective too. 
Next, by \cite{AMRT} Corollary III.7.6, 
we can take a $\Gamma'$-admissible subdivision 
$\Sigma'$ of $\Sigma$ such that 
$X_{\Gamma'}^{\Sigma'} \to X_{\Gamma'}^{\Sigma}$ 
is projective and $X_{\Gamma'}^{\Sigma'}$ is smooth. 
Then $X_{\Gamma'}^{\Sigma'}$ is also projective. 

For any point of $X_{\Gamma}^{\Sigma}$ lying over a $0$-dimensional cusp of $X_{{\G}}^{bb}$, 
the points in its inverse image in $X_{\Gamma'}^{\Sigma'}$ lie over $0$-dimensional cusps of 
$X_{\Gamma'}^{bb}$. 
Therefore, if we could prove Theorem \ref{main} for $X_{\Gamma'}^{\Sigma'}$, 
the assertion for $X_{\Gamma}^{\Sigma}$ follows by pushforward. 
In what follows, we rewrite $(\Gamma', \Sigma')$ as $({\G}, \Sigma)$. 
\end{proof}

Let $P, Q$ be two points of $X^{tor}$ such that 
$p=\pi(P)$ and $q=\pi(Q)$ are $0$-dimensional cusps of $X^{bb}$. 
We first show that $[P]-[Q]\in {\CHtor}$ is torsion in two steps. 

\begin{step}\label{step2}
If $p=q$, we have $[P]=[Q]$ in ${\CHtor}$. 
\end{step}

\begin{proof}
By the description \eqref{eqn: fiber 0dim cusp} of the $\pi$-fiber over $p$, 
we see that $\pi^{-1}(p)$ is a connected configuration of finitely many toric varieties. 
This shows that we can join $P$ and $Q$ by a chain of rational curves inside $\pi^{-1}(p)$. 
\end{proof}

\begin{step}\label{step3}
Even when $p \ne q$, $[P]-[Q]\in {\CHtor}$ is torsion. 
\end{step}

\begin{proof}
We take maximal isotropic subspaces $I$, $I'$ of ${\LQ}$ 
which correspond to $p, q$ respectively. 
We first consider the case when the pairing $(I, I')$ is nondegenerate. 
In that case, we take splittings 
$I=I_1 \oplus \cdots \oplus I_g$ and $I'=I'_1 \oplus \cdots \oplus I'_g$ 
such that 
$(I_i, I'_j)=0$ if $i\ne j$ and $(I_i, I'_i)={\Q}$. 
If we put $\Lambda_{i}=I_i \oplus I'_i$,  
we have the orthogonal splitting 
${\LQ}=\Lambda_{1} \oplus \cdots \oplus \Lambda_{g}$ 
as a symplectic space. 
Note that we have an isomorphism 
$\Lambda_{i}\simeq {\Q}^2$ of symplectic space 
which sends $I_{i}, I'_{i}$ to ${\Q}(1, 0)$, ${\Q}(0, 1)$ respectively. 
This defines an embedding of the upper half plane 
${\HH}=\mathcal{D}_{{\Q}^2}\subset {\proj}^{1}$ as 
\begin{equation*}
\mathcal{D}_{{\Q}^2} \hookrightarrow 
\mathcal{D}_{\Lambda_{1}}\times \cdots \times \mathcal{D}_{\Lambda_{g}} \hookrightarrow {\DL}, 
\qquad V\mapsto V\oplus \cdots \oplus V. 
\end{equation*}
Since ${\G}$ is a congruence subgroup, 
this induces a finite morphism 
$Y(N)\to Y$ 
from the principal congruence modular curve $Y(N)={\HH}/\Gamma(N)$ of some level $N$. 
If $X(N)$ is the compactification of $Y(N)$, 
this extends to a morphism $X(N)\to X^{bb}$ 
which sends the cusps $0, i\infty$ of $X(N)$ to $p, q\in X^{bb}$ respectively. 
Furthermore, since $X(N)$ is a curve, 
this lifts to a morphism $f: X(N)\to X^{tor}$ to the toroidal compactification. 
We have $[P]=[f(0)]$ and $[Q]=[f(i\infty)]$ by Step \ref{step2}. 
By the Manin-Drinfeld theorem,  $[0]-[i\infty]$ is torsion in ${\rm CH}_{0}(X(N))$. 
Sending this equality by $f_{\ast}$, we find that 
\begin{equation*}
[P]-[Q] =  [f(0)] - [f(i\infty)] = f_{\ast}([0] - [i\infty ]) 
\end{equation*}
is torsion in ${\CHtor}$. 
This proves the assertion when the pairing $(I, I')$ is nondegenerate. 

When $(I, I')$ is degenerate, 
we can find another maximal isotropic subspace $I''$ such that 
both $(I, I'')$ and $(I', I'')$ are nondegenerate. 
If we take a point $R\in X^{tor}$ lying over the $I''$-cusp, 
we can use the above case twice to see that 
$[P]-[Q]=([P]-[R])+([R]-[Q])$ is torsion in ${\CHtor}$. 
\end{proof}

\begin{step}\label{step4}
We have $[P]=[Q]$ in ${\CHtor}$. 
\end{step}

\begin{proof}
%By what has been proved, the element $[P]-[Q]$ of ${\CHtor}$ is torsion. 
Let $A_{0}(X^{tor})$ be the subgroup of ${\CHtor}$ of degree $0$ cycles. 
By the theorem of Roitman \cite{Ro2}, the Albanese map 
$A_{0}(X^{tor})\to {\rm Alb}(X^{tor})$ 
is injective on the torsion part of $A_{0}(X^{tor})$. 
On the other hand, the Albanese variety ${\rm Alb}(X^{tor})$ is trivial by the following 
Proposition \ref{prop: Alb trivial}. 
This implies that $A_{0}(X^{tor})$ is torsion-free, 
and so $[P]-[Q]=0$ in ${\CHtor}$. 
This completes the proof of Theorem \ref{main} for Siegel modular varieties. 
\end{proof}

In the final step we used the following general consequence of 
the Margulis normal subgroup theorem (\cite{Margu} p.4). 

\begin{proposition}\label{prop: Alb trivial}
Let ${\D}=G/K$ be a Hermitian symmetric domain and 
${\G}< G$ be an arithmetic group. 
Assume that the Lie group $G$ has real rank $\geq 2$ and ${\G}$ is irreducible. 
Then the Albanese variety of a smooth projective model of ${\D}/{\G}$ is trivial. 
\end{proposition}

\begin{proof}
Let $\tilde{X}$ be a smooth projective model of ${\D}/{\G}$. 
We have a surjective homomorphism 
${\G}\twoheadrightarrow \pi_{1}(\tilde{X})$ (see, e.g., \cite{Sa}). 
This induces a surjective homomorphism  
$\Gamma^{ab}\twoheadrightarrow H_{1}(\tilde{X}, {\Z})$ 
between their abelianizations. 
Then $\Gamma^{ab}$ is finite by the Margulis theorem, 
so $H_{1}(\tilde{X}, {\Z})$ is finite. 
\end{proof}

For Siegel modular varieties, 
Proposition \ref{prop: Alb trivial} also follows from Weissauer's theorem \cite{We}. 
For orthogonal modular varieties, Proposition \ref{prop: Alb trivial} is also proved in \cite{BLMM} when ${\G}$ is congruence,  
using $L^{2}$-cohomology and Lie algebra cohomology. 
This can also be proved in an elementary way by induction on the dimension, 
using restriction to sub orthogonal modular varieties and starting from Hilbert modular surfaces.

\begin{remark}\label{1-cycle}
A similar argument, 
using finiteness of the Mordell-Weil groups of elliptic modular surfaces \cite{Shioda}, 
yields a weaker result for boundary $1$-cycles in the case of Siegel $3$-folds: 
if $Z_{1}, Z_{2}\subset X^{tor}$ are irreducible $1$-cycles such that 
$\pi(Z_{1}), \pi(Z_{2})$ are $1$-dimensional cusps, 
then 
$aZ_{1}\sim bZ_{2}+F$ 
for some $a, b \in {\Z}_{>0}$ 
and a $1$-cycle $F$ whose support is contracted by $\pi$ to points. 
\end{remark} 

\begin{remark}\label{rmk: divisor}
The situation is totally different for the boundary divisors.  
They tend to be linearly independent already in the level of cohomology. 
See \cite{HW} Proposition 4.3 for the boundary divisors over cusps of corank $1$. 
Also, subdivision of fan leads to blow-up (for ${\G}$ neat), 
which just adds new components to $H^{2}$. 
\end{remark}

%%%%Hilbert
\section{Hilbert modular varieties}\label{sec: Hilbert}

Let $F$ be a totally real number field of degree $n>1$. 
The group ${\SL}(F)$ acts on ${\HH}^{n}$ through the natural embedding 
${\SL}(F)\hookrightarrow {\SL}({\R})^{n}$. 
Let ${\G}$ be a finite-index subgroup of ${\SL}(\mathcal{O}_{F})$. 
By \cite{Se}, ${\G}$ is a congruence subgroup. 
The quotient $Y={\HH}^{n}/{\G}$ is a quasi-projective variety called a Hilbert modular variety. 
The Baily-Borel compactification $Y\hookrightarrow X^{bb}$ is obtained by 
adding finitely many points ($0$-dimensional cusps) corresponding to ${\proj}^{1}(F)/{\G}$, 
with the Satake topology (\cite{BB}, \cite{vdG}). 
A toroidal compactification $X^{tor}$ of $Y$ can be constructed by choosing 
a collection of suitable polyhedral cone decompositions of the cone of totally positive elements 
of $F\otimes_{{\Q}}{\R}\simeq {\R}^{n}$, 
one for each cusp independently (\cite{AMRT}, \cite{vdG}). 
We have a natural morphism $X^{tor}\to X^{bb}$. 
Hence $X^{tor}-Y$ is a divisor lying over the finitely many points $X^{bb}-Y$.

\subsection{Proof of Theorem \ref{main}}\label{ssec: Hilbert proof}  

The proof of Theorem \ref{main} for Hilbert modular  varieties 
is similar (and simpler) to the case of Siegel modular varieties. 
Instead of repetition, 
we just indicate that the ingredients in the argument of \S \ref{ssec: Siegel proof}  
have counterparts in the present case. 
We may assume as before that ${\G}$ is neat and $X^{tor}$ is smooth and projective. 

\begin{itemize}
\item The boundary divisor of $X^{tor}$ over a cusp is again a connected configuration of toric varieties. 
\item The diagonal embedding ${\HH}\hookrightarrow {\HH}^{n}$ gives 
a congruence modular curve joining the cusps $0=(0, \cdots, 0)$ and $i\infty=(i\infty, \cdots, i\infty)$. 
Since ${\SL}(F)$ acts transitively on ${\proj}^{1}(F)\times {\proj}^{1}(F) \backslash \textrm{diagonal}$, 
we can translate the diagonal embedding to obtain 
a congruence modular curve joining any two distinct cusps. 
\item Proposition \ref{prop: Alb trivial} applies to the Hilbert modular varieties, 
so the Albanese variety of $X^{tor}$ is trivial. 
In fact, this follows from a more elementary argument (see \cite{vdG} p.82 and p.18).  
\end{itemize}
With these ingredients, 
Theorem \ref{main} for Hilbert modular varieties 
can be proved by the same argument as in the case of Siegel modular varieties. 
\qed 

\begin{example}\label{K3}
In a few examples in $n=2$, the desingularization of $X^{tor}$ is a blown-up $K3$ surface, 
and none of the boundary curves are contracted in the minimal model $X_{K3}$ 
(see \cite{vdG} Chapter VII, \S 3 -- \S 4). 
Since the boundary curves are rational, the unique class in ${\rm CH}_{0}(X_{K3})$ given by 
the boundary points coincides with the Beauville-Voisin class \cite{BV} of the $K3$ surface $X_{K3}$. 
By \cite{BV}, the span of the Beauville-Voisin class contains 
the image of the intersection product ${\rm Pic}(X_{K3})\otimes {\rm Pic}(X_{K3})\to {\rm CH}_{0}(X_{K3})$. 
This provides rational equivalence between boundary $0$-cycles and 
special $0$-cycles in $Y$ obtained as intersection of two modular curves. 
\end{example}

\subsection{Higher modular units}

Let $Y\hookrightarrow X^{tor}$ be a Hilbert modular variety and its toroidal compactification as above. 
Theorem \ref{main} leads to a higher dimensional analogue of modular units (cf.~\cite{KL}) 
as higher Chow cycles on $Y$. 
We write $X=X^{tor}$ and $\partial X= X-Y$. 
We consider a part of the localization exact sequence of higher Chow groups for 
$Y\hookrightarrow X \hookleftarrow \partial X$: 
\begin{equation*}
{\rm CH}_0(X, 1) \stackrel{i^{\ast}}{\to} {\rm CH}_0(Y, 1) \stackrel{\delta}{\to} 
{\rm CH}_0(\partial X) \stackrel{j_{\ast}}{\to} {\rm CH}_0(X). 
\end{equation*}
Here 
$i\colon Y\hookrightarrow X$ and $j\colon \partial X \hookrightarrow X$ 
are natural embeddings and $\delta$ is the connecting map. 
Let $p \ne q \in X^{bb}$ be two cusps. 
We take $P, Q\in X$ such that $\pi(P)=p$ and $\pi(Q)=q$.  
We can view $P$, $Q$ also as $0$-cycles on $\partial X$. 

\begin{corollary}
There exists a nonzero element $Z_{p,q}\in {\rm CH}_{0}(Y, 1)$, 
uniquely determined from $(p, q)$ up to $i^{\ast}{\rm CH}_{0}(X, 1)$, 
such that 
\begin{equation*}
\delta(Z_{p,q})=[P]-[Q] \in {\rm CH}_{0}(\partial X). 
\end{equation*}
\end{corollary}

\begin{proof}
Since $j_{\ast}([P]-[Q])=0$ in ${\rm CH}_0(X)$ by Theorem \ref{main}, 
the existence and uniqueness of $Z_{p,q}$ follow from the exactness of the localization sequence. 
It is independent of the choice of $P$ and $Q$ because 
$[P]=[P']\in {\rm CH}_0(\partial X)$ if $\pi(P)=\pi(P')$ 
(and similarly for $Q$). 
Since $P$ and $Q$ belong to different connected components of $\partial X$, 
we have $[P]\ne [Q]$ in  ${\rm CH}_{0}(\partial X)$ 
and so $Z_{p,q}\ne 0$. 
\end{proof}

The higher Chow cycle $Z_{p,q}$ can be thought of as an analogue of modular units. 
Since it is left somewhat mysterious \textit{how} $P$ and $Q$ are rationally equivalent, 
so is $Z_{p,q}$. 
If $p, q, r\in X^{bb}$ are three cusps, then 
\begin{equation*}
Z_{p,q} + Z_{q,r} \equiv Z_{p,r}  \mod  i^{\ast}{\rm CH}_{0}(X, 1). 
\end{equation*}
Therefore the element  
$Z_{p,q} + Z_{q,r} - Z_{p,r}$ of 
${\rm CH}_{0}(Y, 1)$ provides an element of ${\rm CH}_{0}(X, 1)$.

%%%%Orthogonal
\section{Orthogonal modular varieties}\label{sec: orthogonal}

In \S \ref{ssec: orthogonal proof} we prove Theorem \ref{main} for orthogonal modular varieties, 
after recollection in \S \ref{ssec: orthogonal recall}. 
In \S \ref{ssec: other case} we give an outline for the remaining cases 
(Hermitian and quaternionic modular varieties).

\subsection{Orthogonal modular varieties}\label{ssec: orthogonal recall}

We begin by recalling orthogonal modular varieties and their Baily-Borel and toroidal compactifications. 
In this section we let $\Lambda$ be a free abelian group of rank $2+n$ 
equipped with a nondegenerate symmetric bilinear form $\Lambda \times \Lambda \to {\Z}$ of signature $(2, n)$. 
Let $Q\subset {\proj}\Lambda_{{\C}}$ be the isotropic quadric defined by $(\omega, \omega)=0$. 
The Hermitian symmetric domain ${\D}$ attached to $\Lambda$ 
is one of the two connected components of the open set of 
$Q$ defined by $(\omega, \bar{\omega})>0$.  
Let ${\OL}$ be the index $\leq 2$ subgroup of the orthogonal group ${\rm O}(\Lambda)$ preserving the component ${\D}$. 
Let ${\G}$ be a subgroup of ${\OL}$ of finite index. 
The quotient $Y={\D}/{\G}$ is a quasi-projective variety. 

The domain ${\D}$ has $0$-dimensional and $1$-dimensional rational boundary components, 
corresponding to $1$-dimensional and $2$-dimensional isotropic subspaces $I$ of ${\LQ}$ respectively. 
When $\dim I=1$, the corresponding boundary point is the point $[I_{{\C}}]$ of $Q$. 
The Baily-Borel compactification $X^{bb}$ of $Y$ is defined as the quotient by ${\G}$ 
of the union of ${\D}$ and the rational boundary components, equipped with the Satake topology. 
%It is normal and projective (\cite{BB}). 

A toroidal compactification $X^{tor}$ of $Y$ is construced by choosing 
a finite collection of suitable fans $\Sigma=(\Sigma_{I})$, 
one for each ${\G}$-equivalence class of isotropic lines $I$ in ${\LQ}$ independently. 
%No choice is required for isotropic planes: it is canonical. 
For each $I$, the tube domain realization (= the linear projection from the boundary point) 
defines an isomorphism of ${\D}/{\UIZ}$ with an open set of the algebraic torus 
$T_{I}=U(I)_{{\C}}/{\UIZ}$. 
Here ${\UIZ}={\G}\cap {\UIQ}$ and 
${\UIQ}$ is the unipotent part of the stabilizer of $I$ in ${\rm O}^{+}({\LQ})$. 
${\UIQ}$ is canonically isomorphic to the quadratic space $(I^{\perp}/I)\otimes I$. 
%Then ${\D}/{\UIZ}$ is identified with the locus in $T_{I}$ 
%consisting of points whose imaginary part is in the positive cone of $U(I)_{{\R}}$. 
The fan $\Sigma_{I}$ gives a polyhedral cone decomposition of 
the extended positive cone of $U(I)_{{\R}}$. 
It defines the partial compactification 
${\D}/{\UIZ}\hookrightarrow ({\D}/{\UIZ})^{\Sigma_{I}}$ 
inside the torus embedding 
$T_{I}\hookrightarrow T_{I}^{\Sigma_{I}}$. 
If $\sigma\in \Sigma_{I}$ is a cone which is not an isotropic ray, 
its relative interior is contained in the positive cone, 
and the corresponding torus orbit ${\rm orb}(\sigma)$ in $T_{I}^{\Sigma_{I}}$ 
is also contained in $({\D}/{\UIZ})^{\Sigma_{I}}$. 
The union of these torus orbits ${\rm orb}(\sigma)$ 
is the locus of boundary points of $({\D}/{\UIZ})^{\Sigma_{I}}$ 
which lie over the $I$-cusp (\cite{AMRT} p.164). 

We also have the partial compactifications at the $1$-dimensional cusps, 
which are canonical and require no choice. 
The toroidal compactification $X^{tor}$ is then defined as 
a gluing of $Y$ and natural quotients of these partial compactifications (\cite{AMRT}). 
$X^{tor}$ is a compact Moishezon space with a morphism $\pi:X^{tor}\to X^{bb}$.

\subsection{Proof of Theorem \ref{main}}\label{ssec: orthogonal proof}

We now prove Theorem \ref{main} in the orthogonal case. 
As before we may assume that ${\G}$ is neat and $X^{tor}$ is smooth and projective. 
Let $P, Q$ be two points of $X^{tor}$ such that 
$p=\pi(P)$, $q=\pi(Q)$ are $0$-dimensional cusps, 
corresponding to isotropic lines $I_{1}, I_{2}$ of ${\LQ}$ respectively. 
When $p=q$, we have $[P]=[Q]$ in ${\CHtor}$ 
because $\pi^{-1}(p)$ is a connected union of finitely many toric varieties 
by our description of the partial compactifications. 
In the general case $p \ne q$, it is possible to run a similar argument as before 
when ${\G}$ is a congruence subgroup. 
But instead of such a repetition, we give another proof: 
we reduce the proof to the cases of Siegel $3$-folds and Hilbert modular surfaces 
which were already proved. 
We do not assume that ${\G}$ is a congruence subgroup. 

%By scaling we may assume that $\Lambda$ is even. 

We first consider the case $\Lambda$ has Witt index $2$. 
%By assumption we have $n \geq 3$. 
Then we can find a subspace 
$\Lambda'_{{\Q}}\subset {\LQ}$ containing $I_{1}, I_{2}$ such that 
$\Lambda'_{{\Q}}\simeq 2U_{{\Q}}\oplus \langle -2t \rangle$ for some $t>0$. 
Here $U_{{\Q}}$ is the hyperbolic plane. 
We have $\mathcal{D}_{\Lambda'}\subset {\DL}$, 
and this induces a morphism 
$Y'=\mathcal{D}_{\Lambda'}/\Gamma' \to Y$ 
for some arithmetic subgroup $\Gamma'$ of ${\rm SO}^{+}(\Lambda'_{{\Q}})$. 
This extends to a morphism between the Baily-Borel compactifications, say 
$f \colon (X')^{bb}\to X^{bb}$, 
and also between toroidal compactifications, say 
$f \colon (X')^{tor}\to X^{tor}$,  
for a suitable choice of fans $\Sigma'$ for $\Gamma'$ (cf.~\cite{Ha}). 
Then $I_{1}, I_{2}$ give $0$-dimensional cusps $p', q'$ of $(X')^{bb}$ such that 
$f(p')=p$ and $f(q')=q$ respectively. 
Now $Y'$ is isomorphic to a Siegel modular $3$-fold. 
Therefore we can apply Theorem \ref{main} to $(X')^{tor}$ which is already proved in \S \ref{ssec: Siegel proof}.
This shows that 
$[P']=[Q']$ in ${\rm CH}_{0}((X')^{tor})$ for any points $P', Q'\in (X')^{tor}$ lying over $p', q'$ respectively. 
Sending this equality by $f_{\ast}$, we see that 
\begin{equation*}
[P]=[f(P')]=[f(Q')]=[Q] 
\end{equation*}
in ${\CHtor}$. 
%This proves Theorem \ref{main} for orthogonal modular varieties in the case $\Lambda$ has Witt index $2$. 

In the case $\Lambda$ has Witt index $1$, 
we can instead take a subspace $\Lambda'_{{\Q}}\supset I_{1}, I_{2}$ 
isometric to $U_{{\Q}}\oplus K$ with $K$ anisotropic of signature $(1, 1)$. 
Then $\mathcal{D}_{\Lambda'}/\Gamma'$ for a suitable $\Gamma' < {\rm SO}^{+}(\Lambda'_{{\Q}})$ 
is isomorphic to a Hilbert modular surface, 
for which Theorem \ref{main} is proved in \S \ref{ssec: Hilbert proof}. 
This finishes the proof of Theorem \ref{main} for orthogonal modular varieties.  
\qed

\subsection{Other modular varieties}\label{ssec: other case}

Theorem \ref{main} also holds in the following cases: 
\begin{itemize}
\item The unitary group ${\rm U}(n, n)$ of a Hermitian form of signature $(n, n)$ and Witt index $n$ with $n>1$ 
over an imaginary quadratic field. 
The domain ${\D}$ is of type $I_{n,n}$ (Hermitian upper half space), 
and its compact dual is the Grassmannian $\textrm{G}(n, 2n)$.  
\item The quaternionic group, usually written as ${\rm SO}^{\ast}(4n)$ or sometimes ${\rm Sp}(2n, \mathbb{H})$, 
attached to a skew-Hermitian space of dimension $2n$ and Witt index $n$ with $n>1$ 
over a definite quaternion algebra over ${\Q}$ (cf.~\cite{Kr}). 
The domain ${\D}$ is of type $II_{2n}$ (quaternion upper half space), 
and its compact dual is the orthogonal Grassmannian $\textrm{OG}(2n, 4n)$. 
\end{itemize}
Since we are considering $0$-dimensional cusps, 
the (skew-)Hermitian form must have maximal Witt index, and hence unique up to isometry. 
In both cases, $0$-dimensional cusps correspond to rational maximal isotropic subspaces.  
The Siegel domain realization at a $0$-dimensional cusp is a tube domain 
(Hermitian/quaternion upper half space), 
so the boundary locus of toroidal compactification over that cusp is a union of toric varieties. 
Proposition \ref{prop: Alb trivial} also holds in these cases. 
For joining two $0$-dimensional cusps, we can use (a chain of translates of) diagonal embeddings 
$I_{1,1}\hookrightarrow I_{n,n}$ and $II_{2}\hookrightarrow II_{2n}$ respectively. 
In the spirit of \S \ref{ssec: orthogonal proof}, 
we can use 
\begin{equation*}
III_{n}\hookrightarrow I_{n,n}\hookrightarrow II_{2n}. 
\end{equation*}
We can also use 
$IV_{4}\simeq I_{2,2}\hookrightarrow I_{2n,2n}$ 
and 
$IV_{6}\simeq II_{4}\hookrightarrow II_{4n}$ 
when the dimension is even. 

This exhausts all possible classical domains for which our proof of Theorem \ref{main} may work. 
In general, 
Step \ref{step2} in \S \ref{ssec: Siegel proof} 
works if the unipotent radical of the stabilizer of the $0$-dimensional cusp is abelian. 
Step \ref{step4} works if the Lie group has real rank $>1$ 
and the arithmetic group is irreducible. 
It is not clear if Step \ref{step3} can be done systematically in the general framework of Shimura varieties. 
%In general a chain of several modular curves whose cusps hit to $0$-dimensional cusps will be required. 

\section{Picard modular varieties}\label{sec: Picard}

As explained in \S \ref{sec: intro}, 
Theorem \ref{main} does not hold for Picard modular varieties. 
Besides explicit example (e.g., \cite{Hi}, \cite{Ho}), 
we can tell which part of the argument in \S \ref{ssec: Siegel proof} does not work: 
Step \ref{step2}, as the boundary divisors are abelian varieties; 
Step \ref{step4}, as the Albanese variety of a Picard modular variety is in general nontrivial. 
In this section we show that 
nevertheless we can formulate and prove an analogous result (Theorem \ref{thm: Picard}). 

In \S \ref{ssec: Picard recollect} we recall Picard modular varieties and their compactifications. 
In \S \ref{ssec: special point} we define \textit{special} boundary points. 
In \S \ref{ssec: proof Picard} we prove Theorem \ref{thm: Picard}. 
In \S \ref{ssec: cusp group} we explain that Theorem \ref{thm: Picard} leads to 
a $0$-cycle analogue of cuspidal class groups for Picard modular varieties.

\subsection{Picard modular varieties}\label{ssec: Picard recollect}

Let $K$ be an imaginary quadratic field. 
Let $V$ be a $K$-linear space of dimension $n+1$ 
endowed with a Hermitian form $V\times V \to K$ of signature $(1, n)$. 
We write $V_{{\C}}=V\otimes_{K}{\C}=V\otimes_{{\Q}}{\R}$. 
The space 
\begin{equation*}
{\D} = \{ {\C}v \in {\proj}V_{{\C}} \: | \: (v, v)>0 \} 
\end{equation*}
is isomorphic to the unit ball in ${\C}^n$. 
Let ${\rm U}(V)$ be the unitary group of $V$ and 
${\G}$ be a congruence subgroup of ${\rm U}(V)$. 
We assume, for simplicity of exposition, that ${\G}$ is neat. 
The quotient $Y={\D}/{\G}$ is a smooth quasi-projective variety called a Picard modular variety. 
(It is also sometimes called a \textit{ball quotient}, especially when ${\G}$ is not necessarily arithmetic.) 
Rational boundary points of ${\D}$ correspond to isotropic $K$-lines $I$ in $V$. 
Explicitly, to $I$ corresponds the point $p_{I}=[I_{{\C}}]$ of ${\proj}V_{{\C}}$. 
The Baily-Borel compactification $X^{bb}$ of $Y$ is obtained by adding finitely many points 
corresponding to the ${\G}$-equivalence classes of these points $p_I$.  

Let $I\subset V$ be an isotropic $K$-line. 
We write 
$V(I)=(I^{\perp}/I)\otimes_{K} \bar{I}$ and 
$H(I)={\proj}(V/I)_{{\C}}-{\proj}(I^{\perp}/I)_{{\C}}$. 
Then $H(I)$ is an affine space isomorphic to $V(I)_{{\C}}$ 
(with no specified base point). 
Consider the projection 
\begin{equation*}
\phi \colon {\proj}V_{{\C}} - p_{I} \to {\proj}(V/I)_{{\C}} 
\end{equation*}
from the point $p_{I}$. 
Each point of ${\proj}(V/I)_{{\C}}$, say $[W]$, 
corresponds to a ${\C}$-plane $W\subset V_{{\C}}$ containing $I_{{\C}}$, 
and the $\phi$-fiber over $[W]$ is the affine line ${\proj}W-p_{I}$. 
When $[W]\in H(I)$, the Hermitian form on $W$ is nondegenerate of signature $(1, 1)$. 
This shows that $\phi|_{{\D}}$ realizes ${\D}$ as 
a fibration of upper half planes over $H(I)$, 
with the fiber over $[W]$ being the positive ball ${\D}\cap {\proj}W$ in ${\proj}W$. 
This is the Siegel domain realization of ${\D}$ with respect to $I$. 

Let $N(I)_{{\Q}}\subset {\rm U}(V)$ be the stabilizer of $I$ in ${\rm U}(V)$. 
We write 
\begin{equation*}
W(I)_{{\Q}}={\ker}(N(I)_{{\Q}} \to {\rm GL}(I)\times {\rm U}(I^{\perp}/I)), 
\end{equation*}
\begin{equation*}
U(I)_{{\Q}}={\ker}(N(I)_{{\Q}} \to {\rm GL}(I^{\perp})). 
\end{equation*}
Then $U(I)_{{\Q}}\simeq {\Q}$ is isomorphic to the purely imaginary part of $I\otimes_K \bar{I} \simeq K$.  
Explicitly, if we choose $v\in I$ and isotropic $v'\in V$ with $(v, v')=1$, 
a purely imaginary $\alpha\in K$ corresponds to the isometry 
sending $v'\mapsto v'+\alpha v$ and identity on $I^{\perp}$. 
The group $W(I)_{{\Q}}$ is the unipotent part of $N(I)_{{\Q}}$. 
It is a Heisenberg group with center $U(I)_{{\Q}}$ and 
$W(I)_{{\Q}}/U(I)_{{\Q}}\simeq V(I)$ as ${\Q}$-linear space. 
%so we have the exact sequence 
%\begin{equation*}
%0 \to U(I) \to W(I) \to V(I) \to 0. 
%\end{equation*}
The group $U(I)_{{\Q}}$ acts on $H(I)$ trivially and on the fibers of ${\D}\to H(I)$ by translation. 
The quotient $W(I)_{{\Q}}/U(I)_{{\Q}}$ acts on the affine space $H(I)$ as the translation by $V(I)$. 

Let now ${\G}<{\rm U}(V)$ be a neat congruence subgroup. 
We write 
\begin{equation*}
W(I)_{{\Z}} = {\G}\cap W(I)_{{\Q}}, \quad  
U(I)_{{\Z}} = {\G}\cap U(I)_{{\Q}}, \quad  
V(I)_{{\Z}} = W(I)_{{\Z}}/U(I)_{{\Z}}. 
\end{equation*}
Since ${\G}$ is neat, $W(I)_{{\Z}}$ coincides with ${\G}\cap N(I)_{{\Q}}$. 
$U(I)_{{\Z}}\simeq {\Z}$ is a lattice in $U(I)_{{\Q}}\simeq {\Q}$ and 
$V(I)_{{\Z}}$ is a full lattice in $V(I)$, namely $V(I)_{{\Z}}\otimes_{{\Z}}{\Q}=V(I)$. 
The quotient 
\begin{equation*}
A(I) = V(I)_{{\C}}/V(I)_{{\Z}} 
\end{equation*}
is an abelian variety isogenous to 
the $(n-1)$-fold self product of an elliptic curve with complex multiplication by $K$. 

The toroidal compactification $X^{tor}$ of $Y$ is a canonically constructed 
smooth projective variety with a morphism $\pi\colon X^{tor}\to X^{bb}$. 
No choice of fan is required: it is canonical. 
Explicitly, for each $I$, we consider the punctured disc bundle 
${\D}/U(I)_{{\Z}}$ over $H(I)$ and take the partial compactification, 
say $\overline{{\D}/U(I)_{{\Z}}}$, 
by filling in the origins. 
Then $X^{tor}$ is obtained by gluing $Y$ and 
$\overline{{\D}/U(I)_{{\Z}}}/V(I)_{{\Z}}$ around each cusp $p_{I}$. 
By construction 
the boundary divisor $\Delta(I)=\pi^{-1}(p_I)$ over $p_{I}$ 
is canonically identified with $H(I)/V(I)_{{\Z}}$. 
This is a torsor for the abelian variety $A(I)$ (with no specified origin).

\subsection{Special boundary points}\label{ssec: special point}

Let $I\subset V$ be an isotropic $K$-line. 
The affine space $H(I)={\proj}(V/I)_{{\C}}-{\proj}(I^{\perp}/I)_{{\C}}$ 
is naturally defined over $K$. 
A $K$-rational point of $H(I)$ can be written as 
$[W]=[W_{K}\otimes_{K}{\C}]$ 
for a $K$-plane $W_K\subset V$ containing $I$. 
Since the Hermitian form on $W_K$ is nondegenerate, 
we can write $W_K=\langle I, J \rangle$ 
for another isotropic $K$-line $J\subset V$. 
Hence ${\proj}W$ is the line joining the two rational boundary points $p_I$ and $p_J$. 

If we choose one $K$-rational point as the base point of $H(I)$ and identify $H(I)\simeq V(I)_{{\C}}$ accordingly, 
the set of all $K$-rational points of $H(I)$ is identified with $V(I)\subset V(I)_{{\C}}$. 
The translation by $V(I)_{{\Z}}$ preserves this set. 

\begin{definition}\label{def: special point}
A point of the boundary divisor $\Delta(I)=H(I)/V(I)_{{\Z}}$ is called a \textit{special point} 
if it is the image of a $K$-rational point of $H(I)$. 
\end{definition}

If we choose one special point of $\Delta(I)$ as the origin and identify $\Delta(I)\simeq A(I)$ accordingly, 
the set of all special points of $\Delta(I)$ is identified with the set of torsion points of the abelian variety $A(I)$. 
For $\gamma\in \Gamma$, the induced map 
$\Delta(I)\to \Delta(\gamma I)$ sends special points of $\Delta(I)$ to those of $\Delta(\gamma I)$. 

We have the following geometric characterization of special boundary points. 
We take another isotropic $K$-line $J\ne I \subset V$ 
and put $W=\langle I, J \rangle_{{\C}}$. 
The $1$-dimensional ball ${\D}\cap {\proj}W\subset {\D}$ 
joining $p_I$ and $p_J$ defines a congruence modular curve $Y_W$ and 
a morphism $Y_W\to Y$. 
Let $X_W$ be the compactification of $Y_{W}$. 
Note that $I$ also defines a cusp of $X_W$. 
Since $X_W$ is $1$-dimensional and $X^{tor}$ is compact, 
$Y_W\to Y$ extends to a morphism $X_W\to X^{tor}$. 
Recalling that ${\D}\cap {\proj}W$ is just the $\phi$-fiber over $[W]\in H(I)$, 
we find that the image of the $I$-cusp of $X_W$ in $X^{tor}$ is nothing but the point $[W]\in \Delta(I)$. 
Since every $K$-rational point of $H(I)$ can be written as 
$[\langle I, J \rangle_{{\C}}]$ for some isotropic $K$-line $J\subset V$, 
we obtain the following. 
 
\begin{lemma}\label{lem: special point}
A point of $\Delta(I)$ is special if and only if 
it is the image of the $I$-cusp of $X_W$ by $X_W \to X^{tor}$  
for some $1$-dimensional ball ${\D}\cap {\proj}W$ 
joining $p_I$ and another rational boundary point. 
\end{lemma}

\subsection{Proof of Theorem \ref{thm: Picard}}\label{ssec: proof Picard}

Now we can prove Theorem \ref{thm: Picard}. 
Let $P, Q$ be two special boundary points of $X^{tor}$. 
We want to show that $[P]-[Q]$ is torsion in ${\CHtor}$. 

We first consider the case $\pi(P)=\pi(Q)$, say $p_I$. 
We take $Q$ to be the base point of $\Delta(I)$ and identify $\Delta(I)=A(I)$ accordingly. 
Since $A(I)$ is isogenous to a self product of a CM elliptic curve $E$, 
there exists a homomorphism $f\colon E\to A(I)$ of abelian varieties such that 
$f(o_E)=Q$ and $f(p)=P$ for a point $p$ of $E$ of finite order. 
Since the map 
$E\to {\rm Pic}^{0}(E)$, $x\mapsto [x]-[o_E]$, 
is an isomorphism of abelian groups,
$[p]-[o_E]\in {\rm Pic}^0(E)$ is torsion. 
Thus we find that 
$[P]-[Q]=f_{\ast}([p]-[o_E])$ 
is torsion in ${\CHtor}$. 

Next let $\pi(P)\ne \pi(Q)$. 
Let $\pi(P)=p_I$ and $\pi(Q)=p_J$ 
for isotropic $K$-lines $I, J\subset V$ and 
consider the ${\C}$-plane $W=\langle I, J \rangle_{{\C}}$. 
As explained in \S \ref{ssec: special point}, 
the $1$-dimensional ball 
${\D}\cap {\proj}W$ defines a morphism 
$f\colon X_W\to X^{tor}$ from a compactified congruence modular curve $X_W$. 
Let $q_I$, $q_J$ be the cusps of $X_W$ defined by $I$, $J$ respectively.  
Let $P'=f(q_I)$ and $Q'=f(q_J)$. 
By the Manin-Drinfeld theorem for $X_W$, 
we see that $[P']-[Q']$ is torsion in ${\CHtor}$. 
By Lemma \ref{lem: special point}, $P'$ and $Q'$ are special points in $\Delta(I)$ and $\Delta(J)$ respectively. 
Therefore $[P]-[P']$ and $[Q]-[Q']$ are torsion by the preceding case. 
Thus $[P]-[Q]$ is torsion in ${\CHtor}$. 
Theorem \ref{thm: Picard} is proved. 
\qed  

\vspace{0.5cm} 

By Theorem \ref{thm: Picard}, the \textit{rational} Chow class of special boundary points is uniquely determined. 
%in the rational Chow group ${\CHtorQ}$. 
By Lemma \ref{lem: special point}, 
every intersection point of a modular curve $X_{W}$ and a boundary divisor $\Delta(I)$ 
is a special point. 
Therefore 

\begin{corollary}\label{cor: Chow class special point}
The intersection $0$-cycle of any modular curve and any boundary divisor in the rational Chow group ${\CHtorQ}$ 
is contained in the span of the rational Chow class of special boundary points. 
\end{corollary}

This holds even in the integral Chow group ${\CHtor}$ 
when the first Betti number of $X^{tor}$ happens to be zero, 
by the Roitman theorem \cite{Ro2}.

Of course toroidal compactification can be constructed even when ${\G}$ is not necessarily neat. 
In that general case, the boundary divisor over a cusp $p_{I}$ is the quotient of 
$H(I)/V(I)_{{\Z}}$ by the finite abelian group $({\G}\cap N(I)_{{\Q}})/W(I)_{{\Z}}$. 
Special boundary points can be defined similarly, as the image of $K$-rational points of $H(I)$. 
Obviously Theorem \ref{thm: Picard} still holds in this case.

\begin{example}\label{K3 II}
By Holzapfel \cite{Ho}, there is an example of ${\G}$ in $n=2$ 
such that the minimal desingularization of $X^{tor}$ is a $K3$ surface $X_{K3}$ (of Picard number $20$ and discriminant $3$). 
By Roitman's theorem for $X_{K3}$, any two special boundary points in $X_{K3}$ are rationally equivalent. 
Then the Chow class of these points coincides with the Beauville-Voisin class of $X_{K3}$. 
This can be seen either by considering intersection cycle of modular curve and boundary curve, 
or from the fact that the boundary contains a rational component. 
\end{example}

\subsection{Cuspidal Chow class group}\label{ssec: cusp group}

Theorem \ref{thm: Picard} motivates to consider the following Chow-theoretic analogue of 
cuspidal divisor class groups (cf.~\cite{KL}) for Picard modular varieties. 
(A similar group is trivial for other modular varieties as in Theorem \ref{main}.) 
%(which do not survive for other modular varieties as in Theorem \ref{main}).  

Let $\mathscr{C}$ be the subgroup of ${\CHtor}$ generated by 
the differences $[P]-[Q]$ of two special boundary points. 
By Theorem \ref{thm: Picard}, $\mathscr{C}$ is a torsion group. 
In order to describe $\mathscr{C}$ more explicitly, 
let $X^{tor}-Y=\Delta_1+\cdots + \Delta_N$ be the irreducible decomposition of the boundary. 
We choose a special point $P_i\in \Delta_i$ for each $i$ and  
let $\alpha_i\colon X^{tor}\to \textrm{Alb}(X^{tor})$ be the Albanese map with base point $P_i$. 
Then $A_i=\alpha_i(\Delta_i)$ is a sub abelian variety of $\textrm{Alb}(X^{tor})$. 
Let $(A_i)_{tor}$ be the torsion part of $A_i$. 
By Roitman's theorem \cite{Ro2}, the Albanese map induces an isomorphism 
\begin{equation*}
\mathscr{C} \to \langle \: (A_i)_{tor}+\alpha_1(P_i) \: \rangle_{i} \; \; \subset \; \textrm{Alb}(X^{tor}). 
\end{equation*}
Therefore, to understand $\mathscr{C}$ is essentially reduced to 
understanding the configuration of the shifted sub abelian varieties 
$A_i+\alpha_1(P_i)$ in $\textrm{Alb}(X^{tor})$. 
Note that the subgroup 
$\langle (A_i)_{tor}+\alpha_1(P_i)  \rangle_{i}$ 
of $\textrm{Alb}(X^{tor})$ can also be written as   
$(\sum_{i} A_i)_{tor} +\langle \alpha_1(P_i) \rangle_i$. 
Thus the sub abelian variety $\sum_{i} A_i$ of $\textrm{Alb}(X^{tor})$ gives the divisible part of $\mathscr{C}$ 
(isomorphic to $({\Q}/{\Z})^a$), and the subgroup of $\textrm{Alb}(X^{tor})/\sum_{i}A_{i}$ 
generated by the torsion points $\alpha_{1}(P_i)$, $1\leq i \leq N$, gives the finite abelian part of $\mathscr{C}$.


\begin{thebibliography}{99}

\bibitem{AMRT}Ash, A.; Mumford, D.; Rapoport, M.; Tai, Y. 
\textit{Smooth compactifications of locally symmetric varieties.} 2nd edition. 
Cambridge Univ. Press, 2010. 

\bibitem{BB}Baily, W. L., Jr.; Borel, A.
\textit{Compactification of arithmetic quotients of bounded symmetric domains.} 
Ann. of Math. (2) \textbf{84} (1966), 442--528. 

\bibitem{BLS}Bass, H.; Lazard, M.; Serre, J.-P. 
\textit{Sous-groupes d'indices finis dans ${\rm SL}(n,{\Z})$.} 
Bull. Amer. Math. Soc. \textbf{70} (1964) 385--392. 

\bibitem{BV}Beauville, A.; Voisin, C. 
\textit{On the Chow ring of a $K3$ surface.}
J. Alg. Geom. \textbf{13} (2004), 417--426. 

\bibitem{BLMM}Bergeron, N.; Li, Z.; Millson, J.; Moeglin, C. 
\textit{The Noether-Lefschetz conjecture and generalizations.} 
Invent. Math. \textbf{208} (2017), no. 2, 501--552. 

%\bibitem{dCS}Di Cerbo, L.; Stover, M. 
%\textit{Classification and arithmeticity of toroidal compactifications with $3\bar{c}_2=\bar{c}_1^2=3$.} 
%Geom. Topology \textbf{22} (2018), 2465--2510. 

\bibitem{Dr}Drinfeld, V.~G. 
\textit{Two theorems on modular curves.} 
Funct. Anal. Appl. \textbf{7} (1973), no.2, 155--156. 

%\bibitem{FC} Faltings, G.; Chai, C. 
%\textit{Degeneration of abelian varieties.} 
%Springer, 1990. 

\bibitem{vdG}van der Geer, G. 
\textit{Hilbert modular surfaces.} 
Springer, 1988.

%\bibitem{GHS}Gritsenko, V.; Hulek, K.; Sankaran, G. K. 
%\textit{Abelianisation of orthogonal groups and the fundamental group of modular varieties.} 
%J. Algebra \textbf{322} (2009), no. 2, 463--478.

\bibitem{Ha}Harris, M. 
\textit{Functorial properties of toroidal compactifications of locally symmetric varieties.} 
Proc. London Math. Soc. \textbf{59} (1989), 1--22. 

\bibitem{Hi}Hirzebruch, F. 
\textit{Chern numbers of algebraic surfaces: an example.} 
Math. Ann. \textbf{266} (1984), no. 3, 351--356. 
 
\bibitem{HW}Hoffman, J.W.; Weintraub, S. H. 
\textit{The Siegel modular variety of degree two and level three.} 
Trans. Amer. Math. Soc. \textbf{353} (2000), 3267--3305. 

\bibitem{Ho}Holzapfel, R.P. 
\textit{Chern numbers of algebraic surfaces -- Hirzebruch's examples are Picard modular surfaces.} 
Math. Nachr. \textbf{126} (1986), 255--273. 

%\bibitem{Ko}Kond\=o. S. 
%\textit{On the Albanese variety of the moduli space of polarized K3 surfaces.} 
%Invent. Math. \textbf{91} (1988), no. 3, 587--593. 

\bibitem{KL}Kubert, D.; Lang, S. 
\textit{Modular units.} 
Springer, 1981. 

\bibitem{Kr}Krieg, A. 
\textit{Modular forms on half-spaces of quaternions.}  
Lecture Notes in Math. \textbf{1143}, Springer, 1985. 

\bibitem{Ma}Ma, S. 
\textit{Rational equivalence of cusps.} 
Ann. K-Theory \textbf{5} (2020), no.3, 395--410. 

\bibitem{Manin}Manin, J.~I. 
\textit{Parabolic points and zeta functions of modular curves.} 
Math. USSR Izv. \textbf{6} (1972), no.1, 19--64. 

\bibitem{Margu}Margulis, G. A. 
\textit{Discrete subgroups of semisimple Lie groups.}  
Springer, 1991. 

\bibitem{Me}Mennicke, J. 
\textit{Zur Theorie der Siegelschen Modulgruppe.} 
Math. Ann. \textbf{159} (1965), 115--129. 

%\bibitem{Ro1}Roitman, A.~A. 
%\textit{Rational equivalence of zero-dimensional cycles.} 
%Math. USSR-Sb. \textbf{18} (1974), 571--588. 

\bibitem{PS}Piatetskii-Shapiro, I. 
\textit{Geometry of classical domains and the theory of automorphic functions.}
Gordon and Breach, 1969. 

\bibitem{Ro2}Roitman, A.~A. 
\textit{The torsion of the group of 0-cycles modulo rational equivalence.} 
Ann. of Math. (2) \textbf{111} (1980), no. 3, 553--569. 

\bibitem{Sa}Sankaran, G.~K. 
\textit{Fundamental group of locally symmetric varieties.} 
Manuscripta Math. \textbf{90} (1996), no. 1, 39--48. 

\bibitem{Se} Serre, J.-P. 
\textit{Le probl\`eme des groupes de congruence pour SL2.} 
Ann. of Math. (2) \textbf{92} (1970), 489--527. 

%\bibitem{Shimura}Shimura, G. 
%\textit{Automorphic forms and the periods of abelian varieties.} 
%J. Math. Soc. Japan \textbf{31} (1979), no. 3, 561--592. 

\bibitem{Shioda}Shioda, T. 
\textit{On elliptic modular surfaces.} 
J. Math. Soc. Japan \textbf{24} (1972), 20--59. 

\bibitem{We}Weissauer, R. 
\textit{Vektorwertige Siegelsche Modulformen kleinen Gewichtes.}
J. Reine Angew. Math. \textbf{343} (1983), 184--202. 

\end{thebibliography}
\end{document}